\documentclass{article}
\usepackage[utf8]{inputenc}
\usepackage{amsfonts,amsmath,amssymb,amsthm,booktabs,url}
\usepackage{xcolor}
\usepackage{graphicx}
\usepackage{bbm}
\usepackage{url}

\renewcommand{\le}{\leqslant}
\renewcommand{\ge}{\geqslant}

\renewcommand{\emptyset}{\varnothing}

\newcommand{\natu}{\mathbb{N}}
\newcommand{\ints}{\mathbb{Z}}
\newcommand{\real}{\mathbb{R}}
\newcommand{\N}{\mathbb{N}}

\newcommand{\zb}{\mathbb{Z}_b}

\newcommand{\bsa}{\boldsymbol{a}}
\newcommand{\bsb}{\boldsymbol{b}}
\newcommand{\bsf}{\boldsymbol{f}}
\newcommand{\bsk}{\boldsymbol{k}}
\newcommand{\bsm}{\boldsymbol{m}}
\newcommand{\bsr}{\boldsymbol{r}}
\newcommand{\bsx}{\boldsymbol{x}}
\newcommand{\bsz}{\boldsymbol{z}}
\newcommand{\bsy}{\boldsymbol{y}}
\newcommand{\bsg}{\boldsymbol{g}}

\newcommand{\bsT}{\boldsymbol{T}}
\newcommand{\bsX}{\boldsymbol{X}}
\newcommand{\hk}{\mathrm{HK}}

\newcommand{\bsone}{\boldsymbol{1}}
\newcommand{\bszero}{\boldsymbol{0}}
\newcommand{\rd}{\,\mathrm{d}}
\newcommand{\dnorm}{\mathcal{N}}
\newcommand{\dunif}{\mathbb{U}}

\newcommand{\tran}{\mathsf{T}}
\newcommand{\ind}{\mathbbm{1}}

\newcommand{\vol}{\mathrm{VOL}}
\newcommand{\cov}{\mathrm{Cov}}
\newcommand{\e}{\mathbb{E}}
\newcommand{\wh}{\widehat}

\DeclareMathOperator{\rank}{rank}

\newtheorem{lemma}{Lemma}
\newtheorem{theorem}{Theorem}
\newtheorem{corollary}{Corollary}
\newtheorem{proposition}{Proposition}
\newtheorem{construction}{Construction}
\newtheorem{observation}{Observation}

\theoremstyle{definition}
\newtheorem{definition}{Definition}

\newcommand{\peter}[1]{\begingroup\color{blue}#1\endgroup}

\newcommand{\art}[1]{\begingroup\color{blue}#1\endgroup}

\title{Computable error bounds for 
quasi-Monte Carlo using points with non-negative local discrepancy}
\date{April 2024}
\author{
Michael Gnewuch\\
University of Osnabr\"uck
\and
Peter Kritzer\\
Austrian Academy of Sciences
\and
Art B. Owen\\
Stanford University
\and
Zexin Pan\\
Stanford University
}
\begin{document}
\maketitle

\begin{abstract}
Let $f:[0,1]^d\to\real$ be a completely monotone integrand
as defined by Aistleitner and Dick (2015)
and let points $\bsx_0,\dots,\bsx_{n-1}\in[0,1]^d$
have a non-negative local discrepancy (NNLD) everywhere
in $[0,1]^d$.
We show how to use these properties to get
a non-asymptotic and 
computable upper bound for the integral of $f$
over $[0,1]^d$.  An analogous non-positive local
discrepancy (NPLD) property provides a computable lower
bound. It has been known since Gabai (1967)
that the two dimensional Hammersley points
in any base $b\ge2$
have non-negative local discrepancy. Using
the probabilistic notion of associated random
variables, we generalize
Gabai's finding to digital nets in any base $b\ge2$
and any dimension $d\ge1$ when the generator matrices
are permutation matrices.  We show that permutation
matrices cannot attain the best values of the digital net
quality parameter when $d\ge3$.  As a consequence the
computable absolutely sure bounds we provide come
with less accurate estimates than the usual digital
net estimates do in high dimensions. We are also able to construct high dimensional rank one lattice rules that are NNLD. We show that those lattices do not have good discrepancy properties: any lattice rule with the NNLD property in dimension $d\ge2$ either fails to be projection regular or has all its points on the main diagonal. Complete monotonicity is a very strict requirement that for some integrands can be mitigated via a control variate.
\end{abstract}


\section{Introduction}

Quasi-Monte Carlo (QMC) sampling \cite{dick:pill:2010,nied92} 
can have much better
asymptotic accuracy than plain Monte Carlo (MC), but
it does not come with the usual statistical error
estimates that MC has.  
Those estimates can be recovered
by randomized QMC (RQMC) \cite{lecu:lemi:2002,Owen:2023practical} based on independent replicates
of QMC.  In this paper we consider an alternative
approach to uncertainty quantification for QMC.
For some special sampling points with a non-negative local 
discrepancy (NNLD) property described later
and a suitably monotone
integrand $f$, we can compute upper and lower bounds
on the integral $\mu$ of $f$ over the unit cube in $d$ dimensions.
Methods based on random replication can provide
confidence intervals for $\mu$ that attain a desired level
such as 95\% or 99\% asymptotically, as
the number of replicates diverges.
The method we consider attains 100\% coverage for finite $n$.

Unlike the well-known
bounds derived via the Koksma-Hlawka inequality \cite{hick:2014},
these bounds can be computed by practical algorithms. Convex optimization \cite{boyd:vand:2004} has the notion of a certificate: a computable bound on the minimum value of the objective function.  The methods we present here provide  certificates for multidimensional integration of a completely monotone function.

This improved uncertainty quantification 
comes at some cost.  Our versions of the method
will be more accurate than MC for dimensions $d\le 3$,
as accurate as MC (apart from logarithmic factors)
for $d=4$ and less accurate than MC for $d\ge5$.
They also require some special knowledge of the integrand.

The problem is trivial and the solution is 
well known for $d=1$.
If $f:[0,1]\to\real$ is nondecreasing then
\begin{align}\label{eq:onedimcase}
\frac1n\sum_{i=0}^{n-1}f\Bigl(\frac{i}n\Bigr)
\le \int_0^1f(x)\rd x
\le 
\frac1n\sum_{i=1}^{n}f\Bigl(\frac{i}n\Bigr).
\end{align}
These bracketing inequalities hold even if some
of the quantities in them are $\pm \infty$.
This works because $f$ is nondecreasing, 
the evaluation points in the left hand side are 
`biased low' and those in the right hand side are `biased high'.

To get a multivariate version of~\eqref{eq:onedimcase},
we generalize the notion of points biased low to
points biased towards the origin in terms of a
non-negative local discrepancy (NNLD) property
of the points.  This property was shown to hold
for two dimensional Hammersley points 
by Gabai \cite{gaba:1967} in 1967.
We couple the NNLD property with
a multivariate notion of monotonicity called
complete monotonicity \cite{aist:dick:2015}.

This paper is organized as follows.
Section~\ref{sec:defs} gives some notation and then defines the properties of point sets and functions that we need. Theorem~\ref{thm:basic} there establishes the bracketing property we need. Section~\ref{sec:nnldproperties} gives fundamental properties of NNLD point sets with an emphasis on projection regular point sets. Only very trivial lattice rules, confined to the diagonal in $[0,1]^d$, can be both projection regular and NNLD. Cartesian products preserve the NNLD property as well as an analogous non-positive local discrepancy property. Section~\ref{sec:kh} compares our bounds to those obtainable from the Koksma-Hlawka inequality. Section~\ref{sec:netconstructions} shows that digital nets whose generator matrices are permutation matrices produce NNLD point sets. Section~\ref{sec:rankone} gives a construction of rank one lattice rules that are NNLD. We conclude with a discussion and some additional references in Section~\ref{sec:disc}.

\section{Definitions and a bound}\label{sec:defs}
Here we define a non-negative local discrepancy (NNLD)
property of the points we use as well as
a complete monotonicity criterion for the integrand.
We then establish bounds analogous to~\eqref{eq:onedimcase}.
First we introduce some notation.

\subsection{Notation}

For integer $b\ge1$, let $\ints_b=\{0,1,\dots,b-1\}$.
The set $\{1,2,\dots,d\}$ of variable indices
is denoted by $[d]$.
For $u\subseteq [d]$, we use $|u|$ for the 
cardinality of $u$
and $-u$ for the complement $[d]\setminus u$,
especially in subscripts and superscripts.
For points $\bsx,\bsz\in[0,1]^d$ and a set 
$u\subseteq [d]=\{1,2,\dots,d\}$
let $\bsx_u{:}\bsz_{-u}$ be the hybrid point with
$j$'th component $x_j$ for $j\in u$ and
$j$'th component $z_j$ for $j\not\in u$.
The singleton $\{j\}$ may be abbreviated to just
$j$ and $-\{j\}$ to $-j$. 
Then $\bsx_{-\{j\}}{:}\bsz_{\{j\}}$
becomes $\bsx_{-j}{:}\bsz_j$.
That is, a tuple
such as $\bsz_{\{j\}}$ is written as $\bsz_j$
(not $z_j$).

The points with all coordinates $0$ or all coordinates
$1$ are denoted by $\bszero$ and $\bsone$ respectively.
When it is necessary to specify their dimension we
use $\bszero_d$ and $\bsone_d$. The notation $\ind\{A\}$
is for an indicator variable equal to $1$ when $A$ is true
and $0$ otherwise.

  For integer $d\ge1$ we will use the following precedence
notion on $[0,1]^d$.
For $\bsx,\bsz\in\real^d$ we say
that $\bsx\le\bsz$
when $x_j\le z_j$ holds for all $j=1,\dots,d$.

\subsection{Non-negative local discrepancy}
A QMC rule is given by a list of points $\bsx_0,\dots,\bsx_{n-1}\in[0,1]^d$ and it yields the estimate
$$
\hat\mu = \frac1n\sum_{i=0}^{n-1}f(\bsx_i)
$$
of $\mu$.
We refer to these points as a point set, $P_n$, though
in any setting where some $\bsx_i$ are duplicated
we actually treat $P_n$ as a multiset, counting
multiplicity of the points.
The local discrepancy of $P_n$ at $\bsz\in[0,1]^d$ is given by
$$
\delta(\bsz) = \delta(\bsz;P_n)=\wh\vol([\bszero,\bsz))-\vol([\bszero,\bsz))
$$
where $\vol$ is the Lebesgue measure and $\wh\vol$
is the empirical measure with
$$
\wh\vol([\bszero,\bsz))=\frac1n\sum_{i=0}^{n-1}\art{\ind}_{\bsx_i\in[\bszero,\bsz)}.
$$
That is, $\vol$ is $\dunif[0,1]^d$ while $\wh\vol$
is $\dunif(P_n)$.
The quantity $D_n^*=\sup_{\bsz\in[0,1]^d}|\delta(\bsz)|$ is 
called the star discrepancy of the point set $P_n$.

\begin{definition}
The point set $P_n$ with points $\bsx_0,\dots,\bsx_{n-1}$
has non-negative local discrepancy (NNLD) if
\begin{align}\label{eq:defnnld}
\delta(\bsz)\ge0
\end{align}
for all $\bsz\in[0,1]^d$.
\end{definition}

A distribution for $\bsx\in\real^d$ is positively lower orthant dependent 
\cite{shak:1982} if 
$$\Pr( \bsx\le\bsz)\ge\prod_{j=1}^d\Pr(x_j\le z_j)$$
for all $\bsz\in\real^d$.
A sufficient condition for NNLD is that the $\dunif(P_n)$
distribution on $[0,1]^d$ is positively lower orthant dependent
and that the marginal distributions $\dunif\{x_{0,j},\dots,x_{n-1,j}\}$ for each $j=1,\dots,d$ are
stochastically smaller than $\dunif[0,1]$.
The random variable $X$ is stochastically smaller
than the random variable $Y$ if
$\Pr( X\le z)\ge\Pr(Y\le z)$ for all $z\in\real$
and in that case we also say that the distribution
of $X$ is stochastically smaller than that of $Y$.
There is a related notion of positive upper orthant
dependence as well as two related notions of 
negative orthant dependence, both upper  and lower.

In one dimension, the points
$0,1/n,\dots,(n-1)/n$ are NNLD.
As mentioned earlier, $n=b^m$ Hammersley points
in base $b\ge 2$ and dimension $d=2$ are NNLD
\cite{gaba:1967}.
Those Hammersley points are constructed as follows.  
For $0\le i<n$
write $i=\sum_{k=1}^ma_i(k)b^{k-1}$ for digits $a_i(k)\in\{0,1,\dots,b-1\}$
and set $i' =\sum_{k=1}^ma_i(m-k+1)b^{k-1}$.
Then the $i$'th such Hammersley point is
$\bsx_i=\bigl(i/n,i'/n\bigr)$
for $i=0,1,\dots,n-1$.  Some further properties of the Hammersley points, related to the work of \cite{gaba:1967}, are given by \cite{declerck:1986}.

We will also make use of a complementary property:
non-positive local discrepancy.  
\begin{definition}
The point set $P_n$ with points $\bsx_0,\dots,\bsx_{n-1}$
has non-positive local discrepancy  (NPLD) if
\begin{align}\label{eq:defpnld}
\delta(\bsz)\le0
\end{align}
for all $\bsz\in[0,1]^d$.
\end{definition}

One of our techniques is to take NNLD points $\bsx_i$ and reflect them to $\bsone-\bsx_i$ to get points that oversample rectangular regions near $\bsone$. In doing so we will need to take care of two issues.  One is that for $d\ge2$, the complement of a hyperrectangle $[\bszero,\bsa)$ under this transformation is not another hyperrectangle. The other is that even for $d=1$, the complement of a half open interval $[0,a)$ is a closed interval $[a,1]$.  

To handle these issues we make two observations below.  First, for an $n$-point set $P_n \subset [0,1]^d$ let us additionally define the local discrepancy with respect to closed boxes:
 $$
 \overline{\delta}(\bsz) = \overline{\delta}(\bsz; P_n) 
 =\wh\vol([\bszero,\bsz])-\vol([\bszero,\bsz]).
 $$

\begin{observation}\label{obs:nnldandclosesdiscrep}
The point set $P_n$ has the NNLD property if and only if 
 \begin{equation}\label{closed_NNLD}
 \overline{\delta}(\bsz)  \ge 0
 \hspace{3ex}
 \text{for all $\bsz\in [0,1]^d$.}
 \end{equation}
This is due to the following reasoning: First, we always have
$\overline{\delta}(\bsz) \ge \delta(\bsz)$ for all $\bsz \in [0,1]^d$. Thus the NNLD property of  $P_n$
implies \eqref{closed_NNLD}. For the converse, we assume that $P_n$ satisfies \eqref{closed_NNLD} and consider two cases. If 
 $z_j = 0$ for some $j\in [d]$ then $\delta(\bsz)=0$. If instead $\min_{j\in[d]}z_j>0$ then
 $$
 \delta(\bsz) = \lim_{\varepsilon \downarrow 0} \overline{\delta}(\bsz - \varepsilon \bsone). 
 $$
 Either way, \eqref{eq:defnnld} holds,
 i.e., $P_n$ is NNLD.
 \end{observation}

 \begin{observation}\label{obs:npldalmostiff}
 The condition
 \begin{equation}\label{closed_NPLD}
 \overline{\delta}(\bsz)  \le 0
 \hspace{3ex}
 \text{for all $\bsz\in [0,1]^d$}
 \end{equation}
 implies that $P_n$ has the NPLD property, since $\delta(\bsz) \le \overline{\delta}(\bsz)$ for all $\bsz \in [0,1]^d$.
As a partial converse, if $P_n\subset[0,1)^d \cup \{\bsone\}$, then the NPLD property also implies condition \eqref{closed_NPLD}.
 Indeed, in that case we have $\overline{\delta}(\bsone) = 0$ and 
 $$
 \overline{\delta}(\bsz) = \lim_{\varepsilon \downarrow 0} \delta(\bsz +  \varepsilon \bsone) \le 0
 \hspace{3ex}\text{for all $\bsz \in [0,1)^d$}.
 $$
Now consider for any $\bsz \in [0,1)^d$ and any $\emptyset \neq u\subsetneq [d]$ the closed anchored box $[\bszero, (\bsz_u{:}\bsone_{-u})]$. Due to $P_n\subset[0,1)^d \cup \{\bsone\}$, it contains exactly the same number of points from $P_n$ as the anchored box $[\bszero, (\bsz_u{:}\bsz^*_{-u})]$,  where $\bsz^*$ is defined by 
 $ z_j^* := \max( \{x_{0,j}, \ldots, x_{n-1,j}\}\setminus \{1\})$
 for $j=1, \ldots,d$ taking $z_j^*=0$ in case it is $\max(\emptyset)$.
Consequently, we have
 $$
 \overline{\delta}(\bsz_u{:}\bsone_{-u}) \le \overline{\delta}(\bsz_u{:}\bsz^*_{-u}) \le 0.
 $$

Hence for  $d=1$ we have equivalence of \eqref{closed_NPLD} and NPLD for all $P_n \subset [0,1]$. 
But if $d\ge 2$, then for arbitrary $P_n \subset [0,1]^d$ not contained in $[0,1)^d \cup \{\bsone\}$ the NPLD property does not necessarily imply condition \eqref{closed_NPLD}, as a trivial example with 
$d=2$, $n=1$, $P_n= \{(1,1/2)\}$ shows: $\delta(\bsz) = - \vol([\bszero,\bsz)) \le 0$ for all $\bsz \in [0,1]^d$,
but $\overline{\delta}((1,1/2)) = 1-1/2 =1/2 >0$.
\end{observation}

For $d=1$ if the points in $\tilde P_n$ 
are $1-x_i$ for the points $x_i$ of $P_n$,
then
$$
\overline{\delta}(z;P_n)+\delta(1-z;\tilde P_n) = 0,
$$
i.e.,
$\overline{\delta}(z;P_n)=-\delta(1-z;\tilde P_n)$ for all $z\in [0,1]$.
Then due to Observations~\ref{obs:nnldandclosesdiscrep} and~\ref{obs:npldalmostiff},  reflections of NNLD points are NPLD points
and vice versa for $d=1$.
For $d>1$, if we have positive discrepancy in a box $[\bszero,\bsa)$ prior to reflection, then after reflection a negative discrepancy is only established for $[0,1]^d\setminus(\bsone-\bsa,\bsone]$. Such sets are not generally of the form $[\bszero,\bsb)$ required for the NPLD property; cf. also the discussion in the last paragraph of Section~\ref{SEC:ARV}.

In addition to reflection, we consider another useful transformation.
Let $\tilde\bsx_i$ be the base $b$ Hammersley points for $i=0,\dots,n-1$ where $n=b^m$ and $d=2$. Then \cite{dick:krit:2006} shows that
\begin{align}\label{eq:npldpts}
\bsx_i = (1/n+\tilde x_{i,1},1-\tilde x_{i,2})
\end{align}
are NPLD.

\subsection{Completely monotone functions}
Here we define completely monotone functions,
describing them in words before giving the formal
definition.  If $\bsx\le\bsz$,
then a completely monotone function can increase
but not decrease if any $x_j$ is replaced by $z_j$.
That is $f(\bsx_{-j}{:}\bsz_j)-f(\bsx)\ge0$
always holds.  Next, the size of this difference
can only be increasing as some other component
$x_k$ is increased to $z_k$, so certain differences of differences must also be non-negative. This condition must hold for anywhere from $1$ to $d$ applications of differencing. The $|u|$-fold differences of differences are alternating sums of the form
$$
\Delta_u(\bsx,\bsz)=
\sum_{v\subseteq u}(-1)^{|u-v|}f(\bsx_{-v}{:}\bsz_v).
$$
Note that the coefficient of $f(\bsx_{-u}{:}\bsz_u)$ 
in $\Delta_u(\bsx,\bsz)$ is positive.

\begin{definition}
The function $f:[0,1]^d\to\real$ is completely monotone if
$\Delta_u(\bsx,\bsz)\ge0$
for all non-empty $u$ and all $\bsx,\bsz\in[0,1]^d$
with $\bsx_u\le \bsz_u$.
\end{definition}

In \cite{aist:dick:2015}, Aistleitner and Dick use completely monotone functions to analyze the total variation of $f$ in the sense of Hardy and Krause, denoted by $V_{\hk}(f)$. See \cite{variation} for an account.
From Theorem 2 of \cite{aist:dick:2015}, if
$V_{\hk}(f)<\infty$ then we can write
$$
f(\bsx) = f(\bszero)+f^+(\bsx)-f^-(\bsx)
$$
where $f^+$ and $f^-$ are completely monotone
functions with $f^+(\bszero)=f^-(\bszero)=0$.
They call $f^+-f^-$ the Jordan decomposition
of $f$.  The functions $f^\pm$ are uniquely
determined. 

If $f$ is right-continuous and $V_{\hk}(f)<\infty$
then $f(\bsx)=\nu([\bszero,\bsx])$ for a uniquely
determined signed Borel measure $\nu$, by
Theorem 3 of \cite{aist:dick:2015}. 
Let this signed measure have Jordan decomposition
$\nu=\nu^+-\nu^-$ for ordinary (unsigned) Borel measures
$\nu^\pm$.
Then $f^\pm(\bsx)=\nu^{\pm}([\bszero,\bsx]\setminus\{\bszero\})$. 

The completely monotone functions that
we study take the form
\begin{align}\label{eq:ourcmf}
f(\bsx)=f(\bszero) + \lambda\, \nu([\bszero,\bsx])
\end{align}
where $\nu$ is an arbitrary probability measure on $[0,1]^d$ (or, more precisely, on the Borel $\sigma$-algebra of $[0,1]^d$)
and $\lambda\ge0$. 
Note that every right-continuous completely monotone function $f$ on $[0,1]^d$ can be represented in that way, see, e.g., \cite[II.5.11 Korrespondenzsatz, p.~67]{elstrodt:2018}. 

If $\nu$ is absolutely continuous with respect to the Lebesgue measure, then we may represent $f$, due to the Radon-Nikodym theorem,  as
\begin{align}\label{eq:ourcmf_abs_cont}
f(\bsx)=f(\bszero) + \lambda \int_{[\bszero,\bsx]}
g(\bsz)\rd\bsz
\end{align}
where $g$ is a probability density on $[0,1]^d$, i.e., a non-negative Lebesgue integrable function on  $[0,1]^d$ with integral equal to one.

\subsection{Basic result}

Here we present the basic integration bounds.
To bracket $\mu$ we use up to $2n$ function evaluations
using $n$ each for the lower and upper limits.
For some constructions it is possible that 
some function evaluations might be usable in
both limits, reducing the cost of computation.
For $d=1$ we only need $n+1$ evaluations.

\begin{theorem}\label{thm:basic}
Let $f$ be a completely monotone function of the form
\eqref{eq:ourcmf}.
Let $P_n= \{\bsx_0,\dots,\bsx_{n-1}\} \subset [0,1]^d$, 
and put $\widetilde{P}_n = \{\bsone - \bsx_0,\dots,\bsone - \bsx_{n-1}\}$.
\begin{itemize}
\item[(i)] 
Let $\widetilde{P}_n$ have non-negative
local discrepancy.
Then
\begin{align}\label{eq:upperbound}
\overline\mu =\hat\mu = \frac1n\sum_{i=0}^{n-1}f(\bsx_i)
\ge \int_{[0,1]^d}f(\bsx)\rd\bsx.
\end{align}
\item[(ii)] 
Let $P_n$ have non-positive
local discrepancy.
If additionally either $P_n \subset [0,1)^d \cup \{\bsone\}$ or 
$\nu$ is absolutely continuous with respect to the Lebesgue measure, then
\begin{align}\label{eq:lowerbound}
\underline\mu=\frac1n\sum_{i=0}^{n-1}f(\bsone-\bsx_i)
\le \int_{[0,1]^d}f(\bsx)\rd\bsx.
\end{align}
\end{itemize}
\end{theorem}
\begin{proof}
Without loss of generality take $f(\bszero)=0$
and $\lambda =1$.
Consequently, $f(\bsx) = \nu([\bszero, \bsx])$ for all $\bsx\in [0,1]^d$.
We obtain
\begin{align*}
\mu &= \int_{[0,1]^d} \nu([\bszero, \bsx]) \rd\bsx
 = \int_{[0,1]^d}\int_{[0,1]^d} \ind_{\bsz\le\bsx}
 \rd\nu(\bsz)\rd\bsx.
 \end{align*}
Reversing the order of integration,
\begin{align}\label{eq:mubasic}
\mu&= \int_{[0,1]^d}\int_{[0,1]^d}\ind_{\bsz\le\bsx}\rd\bsx\rd\nu(\bsz)
 =\int_{[0,1]^d}\vol([\bsz,\bsone])\rd\nu(\bsz).
\end{align}
Similarly,
\begin{align*}
\hat\mu
=\frac1n\sum_{i=0}^{n-1}\nu([\bszero, \bsx_i])
=\frac1n\sum_{i=0}^{n-1}\int_{[0,1]^d}\ind_{\bsz\le\bsx_i}\rd\nu(\bsz)
\end{align*}
from which
\begin{align}\label{eq:muhatbasic}
\hat \mu&=\int_{[0,1]^d} \frac1n\sum_{i=0}^{n-1}\ind_{\bsz\le\bsx_i}\rd\nu(\bsz)
=\int_{[0,1]^d}\wh\vol([\bsz,\bsone])\rd \nu(\bsz).
\end{align}
Combining~\eqref{eq:mubasic} and~\eqref{eq:muhatbasic}
the integration error now satisfies
\begin{align}\label{difference_mu_hat_mu}
\hat\mu-\mu&=\int_{[0,1]^d}
\Bigl(\wh\vol([\bsz,\bsone])-\vol([\bsz,\bsone]\Bigr) \rd\nu(\bsz)\notag\\
&=\int_{[0,1]^d}\overline{\delta}(\bsone - \bsz; \widetilde{P}_n)\rd\nu(\bsz),
\end{align}
where $\overline{\delta}(\bsone - \bsz; \widetilde{P}_n)$ is the local discrepancy of
$\widetilde{P}_n$ with respect to the anchored closed box $[\bszero, \bsone - \bsz]$. 
Recall that  $\nu$ is a positive measure.

For part
(i), let $\widetilde{P}_n$ have the  NNLD property.  Due to Observation~\ref{obs:nnldandclosesdiscrep}
we have $\overline{\delta}(\bsone - \bsz; \widetilde{P}_n) \ge 0$ for all $\bsz\in [0,1]^d$.
Hence $\hat\mu\ge\mu$, establishing~\eqref{eq:upperbound}.

For part (ii), let $\widetilde{P}_n$ have the NPLD property. If additionally $\widetilde{P}_n \subset [0,1)^d \cup \{\bsone\}$, 
then Observation~\ref{obs:npldalmostiff} ensures that $\overline{\delta}(\bsone - \bsz; \widetilde{P}_n) \le 0$ for all $\bsz\in [0,1]^d$,
establishing $\hat\mu \le \mu$. If instead $\nu$ is absolutely continuous with respect to the Lebesgue measure, then we can replace
$\overline{\delta}(\bsone - \bsz; \widetilde{P}_n)$  in \eqref{difference_mu_hat_mu} by $\delta(\bsone - \bsz; \widetilde{P}_n)$ without changing the integral. Hence we get again $\hat\mu \le \mu$.
In any case, exchanging the roles of $P_n$ and $\widetilde{P}_n$ establishes~\eqref{eq:lowerbound}. 
\end{proof}

Theorem~\ref{thm:basic} provides an upper bound for $\mu$ when sampling from reflected NNLD points.
This bound will approach $\mu$ as $n\to\infty$
if those points also satisfy $D_n^*\to0$ as $n\to\infty$,
cf. \eqref{eq:koksmahlawka}.
To get a lower bound we can use reflected
NPLD points, provided that either $\nu$ is absolutely continuous or those points all belong to $[0,1)^d\cup\{\bsone\}$.
If $d=2$, then the NPLD points could be
those given by equation~\eqref{eq:npldpts}.  

Both bounds in Theorem~\ref{thm:basic} result from local discrepancy properties of reflected points, $\bsone-\bsx_i$. 
As noted in Observation~\ref{obs:npldalmostiff}, reflecting
NNLD points 
does not generally produce NPLD points (and vice versa).
Also, while NNLD and NPLD properties sound similar
we find in Section~\ref{sec:netconstructions}
that NPLD points are not as simple to
construct as NNLD points.  

\subsection{Example}\label{sec:example}

Here is a simple example to illustrate these
bounds.
The integrand is
known to be completely monotone because it is
a multivariate cumulative distribution function (CDF).  
For $\bsx\in[0,1]^2$ we take
\begin{equation}\label{ex:f1}
f(\bsx) = \Pr( X_1\le x_1, X_2\le x_2)
\end{equation}
for $\bsX\sim\dnorm(0,\Sigma)$ with $\Sigma=
\bigl(\begin{smallmatrix}1&\rho\\\rho&1\end{smallmatrix}\bigr)$
using $\rho=0.7$. Due to \eqref{eq:upperbound}, we can compute an upper bound for $\mu=\int_{[0,1]^2}f(\bsx)\rd\bsx$ by sampling at points
$\bsone-\bsx_i$ where $\bsx_i\in[0,1]^2$ are 
the first $n=b^m$ Hammersley points
in any base $b\ge2$.
We can compute a lower bound for $\mu$ 
by first transforming Hammersley
points via
\eqref{eq:npldpts}
to get NPLD points $\bsx_i$ 
and then sampling at $\bsone-\bsx_i$.
Note that the point sets in 
these bounds are not extensible
in that the points for $n=b^m$ are not
necessarily reused for $n=b^{m+1}$. 

Figure~\ref{fig:example1} shows the results for $n=2^m$ and $1\le m\le13$.  Over the given range,
$n(\overline\mu-\underline\mu)$ increases with $n$ while
$n(\overline\mu-\underline\mu)/\log(n)$ decreases with $n$.
The computed upper and lower bounds for $n=2^{13}$ show that
$$0.5618735\le\mu\le 0.5619890.$$
This function is so smooth and the dimension is so
small that comparable accuracy could be attained
by standard low dimensional integration methods
with many fewer function evaluations.  However, these 
computations 
took approximately five seconds
in R on a MacBook Air M2 laptop, using the {\tt mvtnorm}
package \cite{mvtnorm:book,mvtnorm} to compute $f$. 
Because the computation was so inexpensive, any
inefficiency does negligible harm in this case and
a more efficient algorithm would ordinarily come
without guaranteed error bounds.

\begin{figure}[t!]
\centering
\includegraphics[width=.9\hsize]{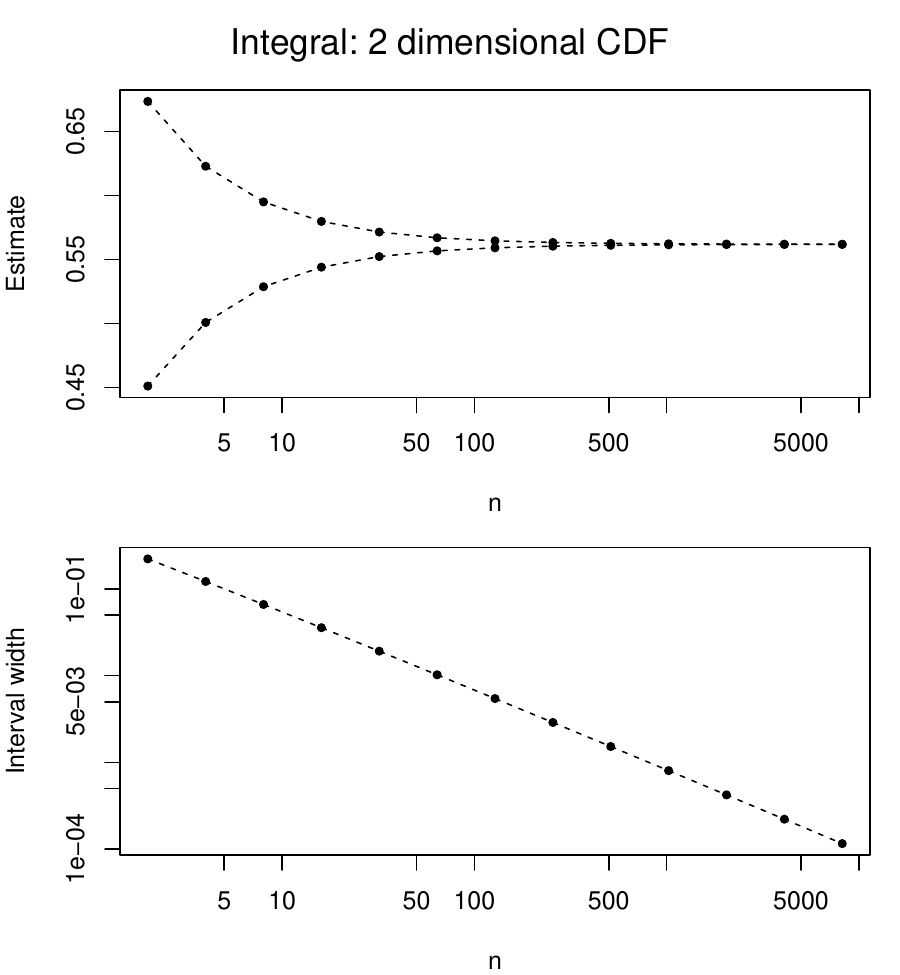}
\caption{\label{fig:example1}
The top panel shows upper and lower bounds for $\mu=\int_{[0,1]^2}f(\bsx)\rd\bsx$ using transformations
of the Hammersley points and $n=2^m$ for $1\le m\le 13$.
The bottom panel plots the difference between those upper
and lower bounds versus $n$, on a logarithmic scale.
}
\end{figure}

\section{More about NNLD points}\label{sec:nnldproperties}

Here we collect some observations about properties
that any $n\ge1$ NNLD points in $[0,1]^d$ 
must necessarily have. Then we use those properties to describe constraints that the NNLD property imposes on customary QMC constructions (lattices and digital nets). Finally we show that the NNLD and NPLD properties are preserved by tensor products.

The first and most obvious property of NNLD points is that
$\bszero$ must be one of those points or else
there is a box $B=[\bszero,\bsa)$ 
with $0=\wh\vol(B)<\vol(B)$ so that $\delta(\bsa)<0$.
Next it must be true that all $n$ points belong
to $[0,1-1/n]^d$. Suppose to the contrary
that $x_{i1}>1-1/n$ for some $0\le i <n$.
Then for some $\epsilon>0$ there exists
$B=[0,1-1/n+\epsilon)\times[0,1]^{d-1}$
with $\wh\vol(B) \le (n-1)/n<\vol(B)$ so that $\bsx_i$
are not NNLD.  The same argument applies if
$x_{ij}>1-1/n$ for any $i$ and any $j$.

Trivial constructions of NNLD points have
$\bsx_i=(i/n)\bsone\in[0,1]^d$
for $0\le i<n$.  We observe that these points
as well as the Hammersley points
for $d=2$ have variables that are positively
correlated.  We will use a general positive
dependence property in Sections~\ref{sec:netconstructions} and~\ref{sec:rankone}
to construct more NNLD point sets.  The NPLD
construction in~\eqref{eq:npldpts} creates a negative lower orthant dependence property for
the components of $\bsx_i\in[0,1]^2$.

Many of the constructions $P_n$ we consider are
projection regular by which we mean that
the projections of $P_n$ onto each single coordinate are
equal to the full set $\{0,1/n,2/n,\dots,(n-1)/n\}$.
Projection regularity is usually considered  advantageous in QMC, as it guarantees a certain structure and even distribution of the integration node set, and simplifies the derivation of error bounds. However, 
combined with the NNLD property, it imposes a
constraint on the point set that we will
use to rule out certain constructions.

\begin{proposition}\label{lem:upper_diag_point}
Let $P_n$ be a point set with $n$ points in $[0,1)^d$ that is projection regular. If $P_n$ has the NNLD property, then 
$P_n$ must contain the point 
$$
\bsx_*=\left(\frac{n-1}{n},\frac{n-1}{n},\ldots,\frac{n-1}{n}\right).
$$ 
\end{proposition}
\begin{proof}
 Suppose that $P_n$ is projection regular and does not contain $\bsx_*$. Then there must exist at least one two dimensional projection $Q_n$ of $P_n$ which does not contain the point $\bsy_*:=(\frac{n-1}{n},\frac{n-1}{n})$. Without loss of generality, assume that $Q_n$ is the projection of $P_n$ onto the first and second coordinates.
 
 This implies, due to projection regularity, that at least two points of $Q_n$ do not lie in the box $[\bszero,\bsy_*)$. Thus,
$$
\delta(\bsy_*) = \wh\vol([\bszero,\bsy_*))-\vol([\bszero,\bsy_*)) \le \frac{n-2}{n}-\frac{(n-1)^2}{n^2}=-\frac{1}{n^2}.
$$
Therefore, $P_n$ has negative local discrepancy for the box
$[\bszero,\bsy_*)\times [0,1)^{d-2}$.
\end{proof}

Proposition~\ref{lem:upper_diag_point} has
some consequences for well known QMC points. We will consider digital nets and integration lattices.
The most widely used and studied integration
lattices are rank one lattices.
Given a generating vector $\bsg=(g_1,\dots,g_d)\in\natu^d$ and a sample size $n\ge1$, a rank one lattice uses points
$$
\bsx_i =\Bigl( \frac{g_1i}n,\frac{g_2i}n,\dots,\frac{g_di}n\Bigr) \bmod 1
$$
for $0\le i<n$ where the modulus operation above takes the
fractional part of its argument.
These $n$ points form a group under addition modulo 1. More general integration lattices having ranks between $1$ and $d$ can also be constructed \cite{dick:krit:pill:2022,  nied92, sloanjoe}. Lattice rules with ranks larger than $1$ are seldom used. They also have the group structure.

\begin{corollary}\label{cor:pro_reg_lattice_NNLD}
 For fixed $d,n \ge 1$ there is only one  projection regular lattice point set in $[0,1)^d$ that consists of $n$ points 
 and has the NNLD property,
namely the lattice point set
$$
\left\{ \bszero, \frac{1}{n} \bsone, \frac{2}{n} \bsone, \ldots, \frac{n-1}{n}\bsone \right\},
$$
whose points all lie on the main diagonal of the $d$-dimensional unit cube $[0,1)^d$.
\end{corollary}
\begin{proof}  
Let $P_n$ be a projection regular lattice point set, consisting of $n$ points in $[0,1)^d$, that has NNLD. Due to Proposition~\ref{lem:upper_diag_point}, $P_n$ has to
contain the point $\bsx_*= \frac{n-1}{n} \bsone$. Due to the additive group structure of $P_n$, we have
$$
k \bsx_* \bmod 1 = \frac{n-k}{n} \bsone \in P_n  \hspace{3ex}\text{for $k=0,1,\ldots, n-1$}.
$$
The set above has $n$ distinct points, so they must be all of $P_n$.
\end{proof}

From Corollary~\ref{cor:pro_reg_lattice_NNLD} we see, in particular, that the only projection regular rank one lattices that are NNLD are trivial, and equivalent to taking all $g_j=1$.
If we also consider lattices that are not projection regular, then we can find constructions that are NNLD and
do not only consist of points on the main diagonal of the unit cube $[0,1)^d$. See Theorem~\ref{thm:non_pro_reg_lattice_NNLD} in Section \ref{sec:rankone}.

Now we look at $(t,m,d)$-nets
\cite{dick:pill:2010,nied92}.
The most widely used $(t,m,d)$-nets are those
of Sobol' in base $b=2$. Sobol' points require one to choose parameters known as direction numbers, with those of \cite{joe:kuo:2008} being especially prominent.
By considering the
point $\bsx_*=\bsone(1-1/n)$, we often find that such Sobol' points
 cannot be NNLD.  The first and third components of $\bsx_i\in[0,1]^d$ for $d\ge3$ 
are projection regular but, for $2\le m\le 20$ they fail to contain $(1-1/n,1-1/n)$.  Therefore the projection of the Sobol' points onto those two dimensions fails to be NNLD and hence the $d$ dimensional point set is not NNLD either.

Like lattice point sets, digital $(t,m,d)$-nets in base $b\ge 2$ have a group structure; this time it is based
on the digitwise addition modulo $b$, which is performed in each component separately.
Using this group structure and Proposition~\ref{lem:upper_diag_point}, we obtain a corollary with a similar flavor
to Corollary~\ref{cor:pro_reg_lattice_NNLD}, although with less dramatic consequences.

\begin{corollary}\label{cor:pro_reg_dig_net_NNLD}
Let $d,m \ge 1$ and $b\ge 2$. Let 
$$
\alpha_{b,m} = \sum_{\nu =1}^m  b^{-\nu} = \frac{1-b^{-m}}{b-1}.
$$
On the one hand, any digital $(t,m,d)$-net in base $b\ge 2$ that is projection regular and has the NNLD property
contains the cyclic subgroup
$$
\{ \bszero, \alpha_{b,m}  \bsone,2 \alpha_{b,m} \bsone, \ldots, (b-1) \alpha_{b,m}\bsone \},
$$
which consists of $b$  points on the main diagonal. 

On the other hand, any $(t,m,d)$-net in base $b\ge 2$ has at most $b^{t+ \lceil \frac{m-t}{d} \rceil}$ points on the main diagonal.
\end{corollary}
\begin{proof}  
Let $n=b^m$, and let $P_n$ be a projection regular digital $(t,m,d)$-net, consisting of $n$ points in $[0,1)^d$, that has NNLD. 
Due to Proposition~\ref{lem:upper_diag_point}, $P_n$ has to
contain the point $\bsx_*= \frac{n-1}{n} \bsone = (b-1)\alpha_{b,m} \bsone$. Using the specific commutative group addition of $P_n$, 
we see that adding up $\bsx_*$ $k$ times yields
$$
k \bsx_*  = (b-k)\alpha_{b,m} \bsone \in P_n$$  
for $k=0,1,\ldots, b-1$.

Now let $P_n$ be an arbitrary $(t,m,d)$-net in base $b$. Put $k:=\lceil \frac{m-t}{d} \rceil$. We may partition the half-open unit cube $[0,1)^d$ into 
$b^{m-t}$ half-open axis-parallel boxes (of the same shape and of volume $b^{t-m}$) with side length $b^{-k}$ and, possibly,  side length  $b^{1-k}$. 
Due to the net property, each of these boxes 
contains exactly $b^t$ points of $P_n$, and at most
$b^k$ of the boxes  have a non-trivial intersection with the main diagonal.
\end{proof}

The next result shows that Cartesian products of
finitely many NNLD (or NPLD) point sets are also NNLD
(respectively NPLD).

\begin{lemma}\label{lem:cartesian}
For positive integers $d_1$, $d_2$, $n_1$ and $n_2$,
let $\bsx_0,\dots,\bsx_{n_1-1}\in[0,1]^{d_1}$
and $\tilde \bsx_0,\dots,\tilde\bsx_{n_2-1}\in[0,1]^{d_2}$.
Let $\bsz_0,\dots,\bsz_{N-1}\in[0,1]^{d_1+d_2}$ for $N=n_1n_2$
be the Cartesian product of those two point
sets.  
If both $\bsx_i$ and $\tilde \bsx_i$ are NNLD points
then $\bsz_0,\dots,\bsz_{N-1}$ are NNLD points.
Analogously, if both $\bsx_i$ and $\tilde \bsx_i$ are NPLD points then
$\bsz_i$ are NPLD points.
\end{lemma}
\begin{proof}
For any $\bsz\in[0,1]^{d_1+d_2}$ 
define $\bsx=\bsz_{[d_1]}$
and $\tilde \bsx=\bsz_{-[d_1]}$.
Let $\vol_1$, $\vol_2$ and $\vol$ denote
Lebesgue measure on $[0,1]^{d_1}$, $[0,1]^{d_2}$
and $[0,1]^d$ for $d=d_1+d_2$, respectively.
Let $\wh \vol_1$, $\wh \vol_2$ and $\wh\vol$
be empirical measures for $\bsx_i$, $\tilde\bsx_i$
and $\bsz_i$ respectively.
If $\bsx_i$ and $\tilde \bsx_i$ are NNLD then
\begin{align*}
\wh\vol([\bszero_d,\bsz))&=
\wh\vol_1([\bszero_{d_1},\bsx))
\wh\vol_2([\bszero_{d_2},\tilde\bsx))\\
&\ge\vol_1([\bszero_{d_1},\bsx))
\vol_2([\bszero_{d_2},\tilde\bsx))\\
&=\vol([\bszero_d,\bsz)).
\end{align*}
Therefore $\delta(\bsz)\ge0$ and $\bsz_i$ are NNLD.
The same argument, with the inequalities reversed,
applies to the NPLD case.
\end{proof}

\section{Comparison to Koksma-Hlawka bounds}\label{sec:kh}

The Koksma-Hlawka inequality is
\begin{align}\label{eq:koksmahlawka}
|\hat\mu-\mu|\le D_n^* V_{\hk}(f)
\end{align}
where $D_n^*$ denotes again the star discrepancy 
and $V_{\hk}(f)$
is the total variation of $f$ in the sense of
Hardy and Krause.
We can be sure that
$$
\hat\mu-D_n^*V_{\hk}(f)\le \mu \le \hat\mu +D_n^*V_{\hk}(f)
$$
but the endpoints of this interval are in general far
harder to compute than $\mu$ is.
One difficulty is that $V_{\hk}(f)$
is a sum of $2^d-1$ Vitali
variations (see \cite{variation}) that in general are harder to compute
than $\mu$ itself is. However when $\tilde{f}$, defined by $\tilde f(\bsx)=f(\bsone-\bsx)$ for every $\bsx$, is completely monotone
then it is useful to work with
an alternative definition of total
variation $V_{\hk\bszero}$ (see \cite{aist:dick:2015}).
For this definition,
$V_{\hk\bszero}(\tilde f) = V_{\hk}(f)$, and 
$V_{\hk\bszero}(\tilde{f})=\tilde{f}(\bsone)- \tilde{f}(\bszero) = f(\bszero)-f(\bsone)$,
 see~\cite{aist:dick:2015}.

With an expression for total variation we still need a value or a bound for $D_n^*$.
The computation of $D_n^*$ is expensive, but
in some instances it might be worth doing,
and for a given set of points we could
pre-compute $D_n^*$.
It is possible to compute $D_n^*$ exactly
at cost $O(n^{d/2+1})$ for fixed $d$ as $n\to\infty$, see \cite{dobk:epps:mitc:1996}.  The cost to compute $D_n^*$ is
exponential in the dimension $d$.
If $n=d\to\infty$ together then computation of $D_n^*$
is NP-complete, see \cite{ gnew:sriv:winz:2008, gia:kna:wah:wer:2012}. Nevertheless, there are algorithms known that provide either upper and lower bounds for $D_n^*$ in moderate dimension, see
\cite{Thie:2001}, or lower bounds for $D_n^*$ even in high dimensions, see \cite{gnew:wahl:winz:2012}. For these and other facts about
computing $D_n^*$, cf. \cite{doer:gnew:wals:2014}.

Then, if we have computed a value
$\varepsilon \ge D_n^*(P_n)$ we then get an interval
$$
\hat\mu \pm \varepsilon (f(\bszero)-f(\bsone))
$$
that is sure to contain $\mu$, when $f(\bsone-\bsx)$ is completely monotone, whether or not $P_n$ is NNLD.

Many of the published bounds for $D^*_n$ with $d\ge2$ take the form
$$
D^*_n \le A_d\frac{\log(n)^{d-1}}n + O\Bigl( \frac{\log(n)^{d-2}}n\Bigr)
$$
where both $A_d<\infty$ and the implied constant within $O(\cdot)$ depend on the dimension $d$ and the construction for the point sets. As a result, they do not provide 100\% bounds that we can use  in a non-asymptotic bound for the integration error.

One notable exception is the result in \cite{gaba:1967} for Hammersley points. For $d=2$ and $n=2^m$ Hammersley points using a base $2$ construction, that paper shows that
\begin{equation}\label{eq:hammersley_exact}
D_n^* = \frac1n\Bigl[ \frac{m}3+\frac19\Bigl(1- \Bigl(-\frac{1}{2}\Bigr)^m\Bigr)\Bigr].
\end{equation}
We also refer to, e.g., \cite{declerck:1986, faur:krit:2013, lar:pil:2003, nied92} and the references therein
for other explicit discrepancy bounds applicable to special cases of (mostly low-dimensional) $(t,m,d)$-nets.

We can use \peter{\eqref{eq:hammersley_exact}} to compare the width of the Koksma-Hlawka interval from~\eqref{eq:koksmahlawka} to the intervals we computed in Section~\ref{sec:example}.

For a function $f$ with a continuous derivative $\partial^2/\partial x_1\partial x_2$ on $[0,1]^2$
\begin{align*}
V_{\hk}(f) &=
\int_0^1\frac{\partial f(x_1,1)}{\partial x_1}\rd x_1
+\int_0^1\frac{\partial f(1,x_2)}{\partial x_1}\rd x_2
+\int_0^1\int_0^1\frac{\partial^2 f(x_1,x_2)}{\partial x_1\partial x_2}\rd x_1\rd x_2\\
&= 3f(1,1)-2f(0,1)-2f(1,0)+f(0,0),
\end{align*}
where we have replaced absolute values of partial
derivatives by the partial derivatives themselves due
to their nonnegativity.
For our example integrand $V_{\hk}(f)\doteq0.7261234$.

When we use the Koksma-Hlawka inequality~\eqref{eq:koksmahlawka} to
get an interval for $\mu$ based on $n$ Hammersley
points, the interval has width
$2D_n^* V_{\hk}(f)$. The intervals we computed in
Section~\ref{sec:example} used $2n$ function evaluations
to get both upper and lower bounds, and so we compare
the width of our 100\% bounds to $2D_{2n}^*V_{\hk}(f)$.
Figure~\ref{fig:wratio} plots the ratio
$$
\frac{\overline\mu_n-\underline\mu_n}{2D_{2n}^*V_{\hk}}
$$
versus $n$ for this example. At the largest sample size, the Koksma-Hlawka bounds are just over 3.6 times as wide as our proposed bands.

\begin{figure}[t]
\centering
\includegraphics[width=.9\hsize]{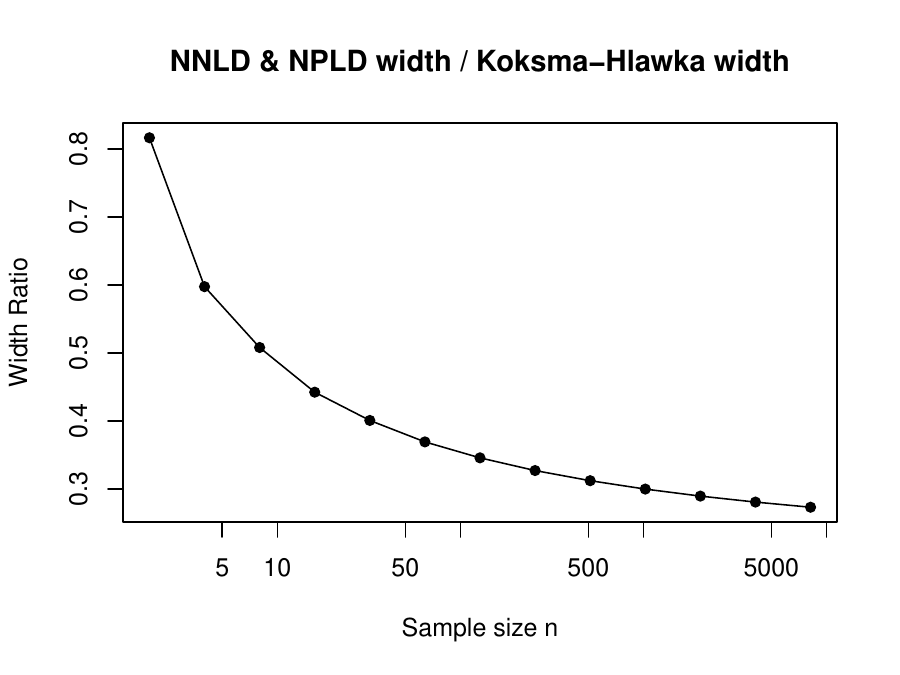}
\caption{\label{fig:wratio}
The horizontal axis has the number $n$ of 
points used in our NNLD and NPLD estimates
for the integrand in Section \ref{sec:example}.
The vertical axis shows the width of a Koksma-Hlawka
interval for $2n$ observations using Hammersley
points divided by the difference between our
upper and lower bounds for $\mu$.}
\end{figure}

\section{Digital net constructions}\label{sec:netconstructions}

The NNLD points of \cite{declerck:1986,gaba:1967} are two dimensional Hammersley points
which are a special kind of digital nets 
\cite{dick:pill:2010} in which the generator matrices
are permutation matrices.  In this section we show that
digital nets constructed with permutation matrices
can be used to get NNLD points with $n=b^m$ points
for any 
integer base $b\ge2$ in any dimension $d\ge1$.
This generalizes the result of \cite{declerck:1986,gaba:1967} which
holds for $d=2$.
We obtain this generalization by a probabilistic
argument using the notion of
associated random variables from reliability theory \cite{esar:pros:walk:1967}.
We also show that there is a limit to how good
digital nets can be when their generator matrices
are permutation matrices.

\subsection{Permutation digital nets}
Here we describe how permutation
digital nets are constructed.
We won't need the more general
definition of digital nets
until we study them more closely in Section~\ref{sec:morenets}.

For a dimension $d\ge1$,  an integer base $b\ge2$
and an integer $m\ge1$ we choose $d$ matrices
$C^{(j)}\in\ints_b^{m\times m}$.
For $n=b^m$ and indices $i=0,1,\dots,n-1$,
write $i=\sum_{k=1}^ma_{i,k}b^{k-1}$ for
$a_{i,k}\in\ints_b$ and put
$\vec{i}=(a_{i,1},\dots,a_{i,m})^\tran$.
Now let $$\vec{x}_{ij}=C^{(j)}\vec{i}\ \bmod b$$
have components $\vec{x}_{ij}(k)\in\ints_b$.
Then $\bsx_i$ has j'th component
$$
x_{ij} = \sum_{k=1}^m\vec{x}_{ij}(k)b^{-k}\in[0,1).
$$
Here we use arithmetic modulo $b$ to define the
digital nets.  It is customary to only use
arithmetic modulo $b$ when $b$ is a prime number
and to use a generalization based on finite fields
when $b=p^r$ for a prime number $p$ and some
power $r\ge2$.  Our proofs of NNLD properties
exploit a monotonicity of integers modulo $b$
whether or not $b$ is a prime.

As an illustration, the 
first $16$ Hammersley points in base $b\ge 2$ 
for $d=2$ are constructed this way with
\begin{align}\label{eq:hammillust}
C^{(1)} = \begin{pmatrix}
1 & 0 & 0 & 0\\
0 & 1 & 0 & 0\\
0 & 0 & 1 & 0\\
0 & 0 & 0 & 1\\
\end{pmatrix}
\quad\text{and}\quad
C^{(2)} = \begin{pmatrix}
0 & 0 & 0 & 1\\
0 & 0 & 1 & 0\\
0 & 1 & 0 & 0\\
1 & 0 & 0 & 0\\
\end{pmatrix}.
\end{align}
Hammersley points for 
$d=2$ and general $m\ge1$ are constructed similarly,
with $C^{(1)}=I_m$ and $C^{(2)}$ a `reversed' identity
matrix as in~\eqref{eq:hammillust}.
The Hammersley points for $d\ge3$ are constructed
using different bases for different components
\cite{hamm:1960}.

The matrices $C^{(j)}$ used to construct digital
nets are called generator matrices.  In the case of the
Hammersley points, we use two generator matrices
and both of them are permutation matrices. 
More generally, digital nets via permutation matrices
are constructed as follows.
For $j\in [d]$, let $\pi_j(\cdot)$ be a permutation of $[m]$. Then define permutation matrices $C^{(j)}$ with elements
$$
C^{(j)}_{rs} = 
\begin{cases}
1, & \pi_j(r)=s\\
0, &\text{else}
\end{cases}
$$
for $r,s\in [m]$.  The $k$'th component of $\vec{x}_{ij}$ is
$$
\vec{x}_{ij}(k)=\sum_{s=1}^m C^{(j)}_{ks}a_{i,s}
= \sum_{s=1}^m \ind_{\pi_j(k)=s}a_{i,s}
= a_{i,\pi_j(k)}.
$$
Therefore
$$
x_{ij} = \sum_{k=1}^mb^{-k}a_{i,\pi_j(k)}.
$$
In the next section we show how to use permutation
generator  matrices to construct NNLD points.

\subsection{Associated random variables}
\label{SEC:ARV}

The settings with $d=1$ or with $n=1$ are trivial 
so we work with $d\ge2$ and $n>1$.
The key ingredient in constructing a short
proof of the NNLD property for our construction
based on a permutation digital net is the
notion of associated random variables \cite{esar:pros:walk:1967}
that originated in reliability theory.

\begin{definition}
Random variables $T_1,\dots,T_m$ are associated if,
for $\bsT=(T_1,\dots,T_m)$ we have
$\cov( g_1(\bsT),g_2(\bsT))\ge0$ for all pairs of functions
$g_1,g_2:\real^m\to\real$ that are nondecreasing in each
argument individually and for which $\e(g_1(\bsT))$,
$\e(g_2(\bsT))$ and  $\e( g_1(\bsT)g_2(\bsT))$ all exist.
\end{definition}

The next theorem uses points that
are a digital net with permutation matrix
generators, followed by shifting every component
of each point to the right by a distance $1/n$.
It shows that they oversample sets of
the form $(\bsz,\bsone]$. Then a reflection
produces NNLD points.

\begin{theorem}\label{thm:perm_implies_nnld}
For integers $m\ge1$, $b\ge2$ and $d\ge2$,
let $\pi_1,\dots,\pi_d$ be permutations of
$\{1,\dots,m\}$, not necessarily distinct.
For $n=b^m$ and $i=0,\dots,n-1$ and $k=1,\dots,m$ define $a_{i,k}\in\ints_b$ via
$i=\sum_{k=1}^m a_{i,k}b^{k-1}$. If $\bsx_i\in(0,1]^d$ has components
\begin{align}\label{eq:shiftedpermnet}
x_{ij} = \frac1n+\sum_{k=1}^mb^{-k} a_{i,\pi_j(k)},\quad j=1,\dots,d
\end{align}
then for any $\bsz\in[0,1]^d$
\begin{equation}\label{eq:perm_implies_nnld}
\frac1n\sum_{i=0}^{n-1}\prod_{j=1}^d\ind\{ x_{ij}>1-z_j\}
\ge \prod_{j=1}^dz_j.
\end{equation}
\end{theorem}
\begin{proof}
We define a random index $i\sim \dunif\{0,1,\dots,n-1\}$
which then implies that for each
index $j$ the digits $a_{i,\pi_j(k)}\sim\dunif(\ints_b)$
independently for $k=1,\dots,m$.
For each $j=1,\dots,d$ we have $x_{ij}\sim\dunif\{
1/n,2/n,\dots,1\}$.  Therefore for any $z_j\in[0,1]$,
$\Pr( x_{ij}>1-z_j)\ge z_j$.

Let $T_j$ be the value of
the random variable $x_{ij}$ where
$i$ is random and $j$ is not.
Letting $\gamma_j$ be the inverse of the permutation
$\pi_j$, we may write
$$T_j=x_{ij}=\frac1n+\sum_{k=1}^mb^{-\gamma_j(k)} a_{i,k}.
$$
Independent random variables $a_{i,k}$ are associated
by Theorem 2.1 of \cite{esar:pros:walk:1967}.
Then $T_1,\dots,T_d$ are associated by 
result (P4) of \cite{esar:pros:walk:1967} because they are 
nondecreasing functions of $a_{i,1},\dots,a_{i,m}$.

For $d=2$, let $g_1(\bsT)=\ind\{x_{i1}>1-z_1\}$
and $g_2(\bsT)=\ind\{x_{i2}>1-z_2\}$.
These are nondecreasing functions of associated random variables
and so by the definition of associated random variables
$$
\Pr( x_{i1}>1-z_1, x_{i2}>1-z_2) \ge \Pr(x_{i1}>1-z_1)\Pr(x_{i2}>1-z_2).
$$
Next, for $2< r\le d$ let
$g_1(\bsT)= \prod_{j=1}^{r-1}\ind\{x_{ij}>1-z_j\}$
and $g_2(\bsT)=\ind\{x_{ir}>1-z_r\}$.
Using induction we conclude that with our random
$i$,
$$
\Pr( x_{ij}>1-z_j,\ j=1,\dots,d)\ge\prod_{j=1}^d\Pr(x_{ij}>1-z_j)
\ge\prod_{j=1}^dz_j
$$
which is equivalent to~\eqref{eq:perm_implies_nnld}.
\end{proof}

\begin{corollary}
For integer $b\ge2$ and dimension $d\ge2$
let $\tilde\bsx_0,\dots,\tilde\bsx_{n-1}\in[0,1]^d$ be points of a
digital net constructed in base $b$
using permutation matrices as generators.
Then the points $\bsx_0,\dots,\bsx_{n-1}\in[0,1]^d$
with $x_{ij} = 1-(1/n+\tilde x_{ij})$ are NNLD.
\end{corollary}
\begin{proof}
Pick $\bsz\in[0,1]^d$.
Now $\ind\{x_{ij}<z_j\}=\ind\{\tilde x_{ij} + 1/n>1-z_j\}$ and so
\begin{align*}
\wh\vol([\bszero,\bsz))=\frac{1}{n}\sum_{i=0}^{n-1}\prod_{j=1}^d\ind\{x_{ij} <z_j\}
=\frac{1}{n} \sum_{i=0}^{n-1}\prod_{j=1}^d\ind\{\tilde x_{ij} +1/n>1-z_j\}
\ge \prod_{j=1}^dz_j
\end{align*}
by Theorem~\ref{thm:perm_implies_nnld}.
\end{proof}

For $d=2$ it was possible to turn an NNLD point set
into an NPLD point set in~\eqref{eq:npldpts} which
includes a reflection $x_{i,2}=1-\tilde x_{i,2}$.
If we were to reflect two or more components of 
an NNLD point set, then those components would
take on a positive upper orthant dependence, which does not 
generally provide the negative
lower orthant dependence we want for NPLD points. For projection regular
NNLD points the reflection of $s\ge2$ components will contain $\bsone_s/n$
and there will be a box $B=[\bszero_s,\bsone_s(1/n+\epsilon))$ with $\delta(B)=1/n-(1/n +\epsilon)^s >0$ for small enough $\epsilon>0$ \art{thus failing to provide NPLD points.}

\subsection{Quality of permutation digital nets}\label{sec:morenets}

It is clear on elementary grounds that a permutation
digital net with two identical permutations among
$\pi_1,\dots,\pi_d$ would have a very high discrepancy.  The resulting
points would satisfy $x_{ij}=x_{ij'}$ for $0\le i<n$ and 
some $1\le j<j'\le d$.
Here we show that our restriction to permutation
digital nets rules out the best
digital nets when $d\ge3$.  We begin with the definitions
of these nets.

\begin{definition}
For integers $d\ge1$, $b\ge2$, and vectors
$\bsk,\bsa\in\natu^d$ with $a_j\in\ints_{b^{k_j}}$
for $j=1,\dots,d$
the Cartesian product
$$
\mathcal{E}(\bsk,\bsa)=\prod_{j=1}^d\Bigl[ \frac{a_j}{b^{k_j}},\frac{a_j+1}{b^{k_j}}
\Bigr)
$$
is an elementary interval in base $b$.
\end{definition}

\begin{definition}
For integers $b\ge2$, $d\ge1$ and $0\le t\le m$, the $n$
points $\bsx_0,\dots,\bsx_{n-1}$ are a $(t,m,d)$-net in
base $b$ if
$$
\wh\vol(\mathcal{E}(\bsk,\bsa))=\vol(\mathcal{E}(\bsk,\bsa))
$$
holds for all elementary intervals in base $b$
for which $\sum_{j=1}^dk_j\le m-t$.
\end{definition}

Digital nets are $(t,m,d)$-nets.  Other parameters being
equal, smaller values of $t$ denote better equidistribution
of the points $\bsx_i$ which translates into a 
lower bound on $D_n^*$ and hence a smaller upper
bound in the Koksma-Hlawka inequality.
From 
Theorem 4.10 of \cite{nied92} 
\begin{align}\label{eq:niedrate}
D_n^* = O\Bigl( \frac{b^t\log(n)^{d-1}}n\Bigr)
+O\Bigl( \frac{b^t\log(n)^{d-2}}n\Bigr)
\end{align}
where the implied constants depend only on $d$
and $b$. The powers of $\log(n)$ are not negligible
but they are often not actually observed in numerical experiments \cite{wherearethelogs}.

The quality parameter of a permutation digital net
can be very bad.  For $d=2$, taking the Hammersley
construction yields $t=0$ which is the best possible
value.  Here we show that for $d\ge3$, the best
available values of $t$ are far from optimal.

The following definition and result are
based on \cite[Sect.~2.3]{mato:1998}.
While digital net constructions are available
based on any finite field, our NNLD constructions
using permutation matrices as generators
require them to be based on modular arithmetic.

\begin{construction}[Digital Construction of $(t,m,d)$-Nets]\label{Cons_6}
For a prime number $b\ge 2$, and $C^{(1)},\dots,C^{(d)} \in (\zb)^{m\times m}$, let $\mathcal{C} = \{ C^{(1)}, \ldots, C^{(d)} \}$.
For $h\in \zb^m$ define $p(h) \in [0,1)^d$ componentwise by its $b$-adic digit expansion
\begin{equation*}
p(h)_j = 
\delta^{(j)}_1(h) b^{-1} + \delta^{(j)}_2 (h) b^{-2} + \cdots + \delta^{(j)}_m(h) b^{-m}
\in [0,1), \hspace{3ex} j=1,\ldots,d,
\end{equation*}
where $\delta^{(j)}(h) = (\delta^{(j)}_1(h), \ldots, \delta^{(j)}_m(h) )$ is simply the vector 
$C^{(j)}h\in \zb^m$ under arithmetic modulo $b$.
We define the point set 
\begin{align}\label{eq:pointset}
P(\mathcal{C}) = (p(h))_{h\in \zb^m}.
\end{align}
Clearly, $|P(\mathcal{C})| = b^m$.

To assess the quality of $P(\mathcal{C})$, we define the \emph{quality criterion $\rho(\mathcal{C})$}: 
For $\bsm=(m_1,m_2, \ldots, m_d)\in \{0,1,\ldots,m\}^d$ 
with $|\bsm|=\sum_{j=1}^dm_j$ let  
$$
\mathcal{C}^{(\bsm)}
= \begin{pmatrix}
C^{(1)}(1{:}m_1, \cdot )\\
C^{(2)}(1{:}m_2,\cdot)\\
\vdots\\
C^{(d)}(1{:}m_d,\cdot)
\end{pmatrix}\in \zb^{|\bsm|\times m}
$$
where $C^{(j)}(1{:}m_j,\cdot)\in\zb^{m_j\times \art{m}}$
represents the
first $m_j$ rows of $C^{(j)}$.
Now $\rho(\mathcal{C})$ is the maximum number $\rho \in \{0,1,\ldots,m\}$ such that for all
$\bsm\in \{0,1,\ldots,m\}^d$ with $|\bsm| = \rho$
we have $\rank (\mathcal{C}^{(\bsm)}) = \rho$
in arithmetic modulo $b$.
\end{construction}

\begin{proposition}\label{Prop_7}
Let $b,m,\mathcal{C}$, and $P(\mathcal{C})$ be as in Construction~\ref{Cons_6}.
Then $P(\mathcal{C})$ is a $(t,m,d)$-net for $t= m - \rho(\mathcal{C})$.
\end{proposition}

\begin{observation}\label{Observaion_2}
The proposition shows that the best possible $t$-value $t(\mathcal{C})$ of $P(\mathcal{C})$ is at most
$m - \rho(\mathcal{C})$.  But similar arguments as in the corresponding proof of \cite[Proposition~2.7]{mato:1998}
show that actually 
$$
t(\mathcal{C}) = m - \rho(\mathcal{C}).
$$    
\end{observation}

\begin{proposition}\label{Prp:V_t_value}
Let $V:= \{v_1, \ldots, v_m \}$ be a set of linearly independent vectors in $\zb^m$.
Let $m= \ell d + r$, where $\ell\in \N_0$ and $0 \le r < d$.
If the rows  $C_k^{(j)}$, $k=1, \ldots,m$, of the matrices $C^{(j)}$, $j=1,\ldots,d$, are all
contained in $V$, then $\rho(\mathcal{C}) \le 2 \lfloor m/d \rfloor + 1$.
Therefore, the smallest $t$-value $t(\mathcal{C})$ of $P(\mathcal{C})$ satisfies
\begin{equation*}
 t(\mathcal{C}) \ge (d-2) \lfloor  m/d \rfloor + r - 1.   
\end{equation*}
\end{proposition}
\begin{proof}
Consider the $m$ row vectors 
$$C^{(1)}_1, C^{(2)}_1, \ldots, C^{(d)}_1, \quad C^{(1)}_2, C^{(2)}_2, \ldots,C^{(d)}_2,\quad\ldots\quad, C^{(1)}_{\ell+1},C^{(2)}_{\ell+1},\ldots,C^{(r)}_{\ell +1}.$$
where $C_i^{(j)}$ is row $i$ of $C^{(j)}$.

\emph{Case 1}: Two of these row vectors are equal. Assume these rows are $C^{(j)}_k$ and $C^{(j')}_{k'}$.
If $j=j'$, then we consider the matrix $C:= \mathcal{C}^{(\bsm)}$ with $m_j = \max\{k, k'\}$
and $m_\nu = 0$ for all $\nu \neq j$. Obviously, $\rank(C) \le \max\{k, k'\} - 1$. Hence it follows
that $\rho(\mathcal{C}) \le \max\{k, k'\} - 1 \le \lceil m/d \rceil - 1$.
If $j\neq j'$, then we consider the matrix $C:= \mathcal{C}^{(\bsm)}$ with $m_j = k$, $m_{j'} = k'$,
and $m_\nu = 0$ for all $\nu \notin \{j,j'\}$. Obviously, $\rank(C) \le k +  k' - 1$. Hence it follows
 that $\rho(\mathcal{C}) \le k + k' - 1 \le 2\lceil m/d \rceil - 1$.

\emph{Case 2}: All of these row vectors are different. 
Consider $C^{(d)}_{\ell+1}$. Then there exist $1\le j <  d$ and $1\le h \le \ell+1$ or
$j = d$ and $1\le h \le \ell$
such that $C^{(d)}_{\ell+1} = C^{(j)}_h$.

Now we argue similarly as in case 1:
If $j=d$, then it is easy to see that $\rho(\mathcal{C}) \le \ell= \lfloor m/d \rfloor$.
If $j\neq d$, then $\rho(\mathcal{C}) \le h+ \ell \le 2\ell + 1 \le 2  \lfloor m/d \rfloor+1$.

In any case, we have shown that $\rho(\mathcal{C}) \le 2  \lfloor m/d \rfloor+1$.
\end{proof}

\begin{corollary}\label{cor:tlowerbound}
Let $m= \ell d + r$, where $\ell\in \N$ and $0 \le r < d$.
If $C^{(1)}, \ldots, C^{(d)} \in \zb^{m\times m}$ are all permutation matrices, then 
the smallest $t$-value $t(\mathcal{C})$ of $P(\mathcal{C})$ satisfies
\begin{equation*}
 t(\mathcal{C}) \ge (d-2) \lfloor  m/d \rfloor + r - 1.   
\end{equation*}
\end{corollary}

\begin{proof}
This follows directly from Proposition~\ref{Prp:V_t_value}, since the rows of the 
matrices $C^{(1)}, \ldots, C^{(d)}$ are all in $\{ e_1, \ldots, e_m\}$, where $e_i$ denotes the $i$-th standard unit vector of $\zb^m$.
\end{proof}

Let us represent the permutation matrix where row $k$ has
a one in column $\pi(k)$ as simply the column
vector with entries $\pi(k)$. 
Then we can represent our permutation nets with an
$m\times d$ matrix $\Pi$ with $j$'th column $\pi_j$.  For
example the Hammersley points 
with generator matrices $I_m$ and reversed $I_m$
are represented this way by
\begin{align}\label{eq:hammerpi}
\Pi=\begin{pmatrix}
1 & m\\
2 & m-1\\
\vdots & \vdots\\
m & 1
\end{pmatrix}.
\end{align}
For $d=3$ we want $\Pi\in\{1,\dots,m\}^{m\times 3}$
with the largest possible value of
$$
\rho=\min\bigl\{ k+k' \mid \Pi_{k,j}=\Pi_{k',j'}, 1\le j<j'\le 3\bigr\}-1.
$$
Then we get quality parameter $t=m-\rho$.
If we simply adjoin a third column to $\Pi$ in~\eqref{eq:hammerpi} the best $\rho$
we can get is $m/2$ if $m$ is even
and $(m+1)/2$ if $m$ is odd. These lead
to $t\ge m/2$ if  $m$ is  even  and $t\ge (m-1)/2$ if $m$  is odd, which is much worse
than the bound in Corollary~\ref{cor:tlowerbound}.
For $t=m/2$ the first term
in~\eqref{eq:niedrate} is 
$O(b^{m/2}\log(n)^2/n)=O(\log(n)^2/\sqrt{n})$
because $b=n^{1/m}$.

If $m=3\ell$, then we can choose the first $\ell$ rows of
$\Pi$ to be
$$
\begin{pmatrix}
1 & 2 & 3\\
4 & 5 & 6\\
\vdots & \vdots &\vdots\\
3\ell-2 &3\ell-1 & 3\ell\\
\end{pmatrix}.
$$
Let us label these first $\ell$ rows of $\Pi$
by $\bsr_1,\bsr_2,\dots,\bsr_\ell\in\natu^3$. 
Now, for $\bsr = (a,b,c)$ let
$\bsr'=(b,c,a)$ and $\bsr''=(c,a,b)$ be one and
two rotations of the elements of $\bsr$ to the
left with wraparound.
By taking the rows of $\Pi$ in this order
$$
\bsr_1,\bsr_2,\dots,\bsr_\ell,
\enspace \bsr_{\ell}',\bsr_{\ell-1}',\dots,\bsr_1',
\enspace \bsr_{\ell}'',\bsr_{\ell-1}'',\dots,\bsr_1''
$$
we get $\rho = 2\ell$ and hence $t=m/3$.
This is very close to the bound $\lfloor m/d\rfloor+0-1=m/3-1$ from Corollary~\ref{cor:tlowerbound}.
We prefer the ordering
$$
\bsr_1,\bsr_2,\dots,\bsr_\ell,\enspace 
\bsr'_\ell,\bsr''_\ell,\enspace 
\bsr_{\ell-1}',\bsr_{\ell-1}'',\enspace
\bsr'_{\ell-2},\bsr''_{\ell-2},\enspace 
\enspace \dots\enspace
\bsr_2',\bsr_2'',\enspace \bsr_1', \bsr_1''
$$
because while it attains the same value of $t$ it has fewer pairs of columns for which $k+k'=2\ell+1$.
With $t=m/3$ for $d=3$ the first term
in~\eqref{eq:niedrate} is
$O(b^t\log(n)^2/n)=O(n^{-2/3}\log(n)^2)$.

Using the same method for $d=4$ and $m=4\ell$
we can get $\rho=2\ell =m/2$, implying that $t=m/2$, and yielding a
rate of $O(b^t\log(n)^3/n)=O(n^{-1/2}\log(n)^3)$.
This result for $d=4$ matches the rate
for plain MC apart from the power of $\log(n)$.
So the 100\% error bounds available from NNLD sampling
come with a logarithmic accuracy penalty
in comparison to plain MC.

A second choice for $d=4$ is to use a Cartesian
product of two Hammersley point sets with $\sqrt{n}$
points each.
The error of such a Cartesian product would
ordinarily be the same as that of the individual
Hammersley rules in two dimensions with their
reduced sample sizes.  That is $O( n^{-1/2}\log(n))$
which is then a better logarithmic factor
than the $4$ dimensional permutation nets attain.

For $d=3$ we could also use a Cartesian product of
Hammersley points with $n=b^2$ points and a one
dimensional grid $\{0,1/n,\dots,1-1/n\}$.  
This then uses $N=n^2$ points and we expect an
error of $O(\log(n)/n)=O(\log(N)/N^{1/2})$ which
is a worse rate than
we can get with the permutation net in $[0,1]^3$.

\subsection{Other generator matrices}

Permutation matrices are not the only generator
matrices that can produce points with the NNLD
property.
For digital nets in base $2$, we know from
Proposition~\ref{lem:upper_diag_point} that if $C^{(1)}=I_m$
then we must have $C^{(j)}\bsone_m=\bsone_m \bmod 2$.
This in turn implies that every row of $C^{(j)}$
must have an odd number of $1$s in it. A numerical search shows there are 221 choice of nonsingular $C^{(2)}$ when $m=4$ and $C^{(1)}=I_4$. Below are some examples:
\begin{align*}
C^{(2)} = \begin{pmatrix}
1 & 0 & 0 & 0\\
1 & 1 & 0 & 1\\
0 & 1 & 1 & 1\\
1 & 1 & 1 & 0\\
\end{pmatrix}
\quad \text{or}\quad
 \begin{pmatrix}
0 & 1 & 0 & 0\\
1 & 0 & 0 & 0\\
1 & 0 & 1 & 1\\
1 & 1 & 1 & 0\\
\end{pmatrix}
\quad \text{or}\quad
 \begin{pmatrix}
0 & 0 & 1 & 0\\
1 & 0 & 0 & 0\\
1 & 1 & 0 & 1\\
0 & 1 & 0 & 0\\
\end{pmatrix}
.
\end{align*}

Nevertheless, it is hard to find an example where non-permutation matrices perform better than permutation matrices with respect to the $t$-value. When $d=3$, one can verify, either by lengthy reasoning or brute-force enumeration, that NNLD digital nets constructed by non-permutation matrices cannot attain a better t-value than those constructed by permutation matrices for $m\le 7$ and $b=2$.

\subsection{Composite values of $b$}

Our constructions using permutation matrices as
generators do not actually require $b$ to be a prime number,
nor does the result of \cite{gaba:1967} on the NNLD property of Hammersley points.  When $b$ is a prime power, using modular arithmetic will not generally give the same nets that the usual algorithms (see \cite{dick:pill:2010}) over finite fields do. Our proofs do not show that nets based on finite fields of non-prime cardinality will lead to NNLD points, although those other nets may have a much better value of $t$.

The integers modulo a composite number $b$ form a commutative ring but they do not form a field, and this complicates the definition of $t$ which relies on matrix rank. Digital net constructions based on such rings are described in \cite{nied92}. The notion of rank in Theorem 4.26 there is that the linear equation $\mathcal{C}^{(\bsm)}\bsz=\bsf$ with $|\bsm|=m-t$ should have exactly $b^t$ solutions $\bsz$ for any vector $\bsf$ in arithmetic over the ring. Then a digital construction provides a point set with the given value of $t$.  Our matrices $C^{(j)}$ are binary with exactly one $1$ in each row.  Then this more general notion of rank becomes the number of distinct rows in $\mathcal{C}^{(\bsm)}$, the same as it would be for a field.
\section{Non-trivial Rank 1 lattices that are NNLD}\label{sec:rankone}

Here we consider special cases of rank-1 lattice rules that are
suboptimal in terms of discrepancy, but produce
NNLD points.  While they can be defined in any dimension $d\ge2$ it is only for dimension $1$ that they are projection regular. Therefore the conclusions from Proposition~\ref{lem:upper_diag_point}
and Corollary~\ref{cor:pro_reg_lattice_NNLD} do not hold for them when $d>1$.

\begin{theorem}\label{thm:non_pro_reg_lattice_NNLD}
For integers $m\ge d$ and $b\ge2$ 
and $0\le i<n=b^m$, let 
$$
\bsx_{i} = 
\Bigl(
\frac{i}{n},
\frac{ib}{n},
\dots,
\frac{ib^{j-1}}n,
\dots,
\frac{ib^{d-1}}{n}
\Bigr) \quad\bmod 1.$$
Then the points $\bsx_0,\dots,\bsx_{n-1}$
are NNLD.
\end{theorem}
Before proving this theorem we note that these points are quite poor for integration; however, the structure of the points can be useful for showing good integration bounds in suitably weighted spaces, see \cite{DKLP15}.  There are only $b^{m-j+1}$ unique values of $x_{ij}$. Further, when $|j-j'|$ is small the points $(x_{ij},x_{ij'})$ lie within at most $b^{|j-j'|}$ lines in $[0,1)^2$ and have a large discrepancy.

\begin{proof}
We write $i=\sum_{k=1}^ma_i(k)b^{k-1}$ and then
\begin{align*}
nx_{ij} &= b^{j-1}\sum_{k=1}^ma_i(k)b^{k-1}
\ \bmod b^m
= \sum_{k=1}^{m+1-j}a_i(k)b^{j+k-2}.
\end{align*}
For $i\sim\dunif\{0,1,\dots,n-1\}$
the digits $a_i(1),\dots,a_i(m)$
are independent $\dunif(\ints_b)$ random
variables.  Hence they are associated random
variables which makes $nx_{i1},\dots,nx_{id}$
and hence $x_{i1},\dots,x_{id}$ into
associated random variables.
Finally, $x_{ij}$ has the uniform distribution 
on $\{0,1/n_j,2/n_j,\dots,1-1/n_j\}$
where $n_j=n/b^{j-1}$. This distribution is
stochastically smaller than $\dunif[0,1]$
and so $\bsx_i$ are NNLD.
\end{proof}

The values $x_{ij}$ for $0\le i<b^m$ in these lattices take $n_j=b^{m-j+1}$ distinct values $\ell/n_j$ for $0\le \ell<n_j$ with each of those values appearing $n/n_j$ times. As such they constitute a left endpoint integration rule on $n_j$ points and so for nonperiodic smooth integrands we anticipate an error rate of $O(n_j^{-1})$. For this to be better than plain MC we require $n_j\ge\sqrt{n}$ or $j\le m/2$. While a better rate is available for periodic integrands, those cannot be completely monotone unless they are constant.

\section{Discussion and further references}\label{sec:disc}

We find that it is possible to get computable 
bounds on some integrals by using points with a suitable
bias property (non-negative local discrepancy (NNLD)) on
integrands with a suitable monotonicity property (complete monotonicity).  The method of associated random variables
is useful for showing that a given point set is NNLD.

There are several generalizations of multivariate monotonicity
in \cite{mosl:scar:1991}. They include the complete
monotonicity discussed here as well as the more commonly 
considered monotonicity in each of the $d$ inputs one
at a time. The complexity of integrating coordinate-wise monotone functions has been studied by \cite{novak:1992,papa:1993}.
Scrambled $(t,m,d)$-nets have been shown to be
negatively orthant dependent if and only if $t=0$
\cite{wiar:lemi:dong:t2021}.  
Similarly, it was shown in \cite{WnuGne:2020} that randomly shifted and jittered (RSJ) rank-$1$ lattices based on a random
generating vector are also negatively orthant dependent and that, in some sense, one cannot achieve this result by employing less randomness.
Using the
NLOD property of the distribution of these
RQMC points, it follows from \cite{lemi:2018} that for functions which are monotone in each variable
scrambled nets and RSJ rank-1 lattices cannot increase variance
over plain Monte Carlo in  any dimension $d$.  

While complete monotonicity is a very special property,
its applicability can be widened by the method of
control variates.  If $h(\cdot)$ is completely monotone
with known integral $\theta$, we will in some settings be able to find $\lambda_+>0$ for
which $f+\lambda_+ h$ is a completely monotone
function of $\bsx$. Then by Theorem~\ref{thm:basic} we
can compute an upper bound $B_+\ge\mu+\lambda_+\theta$
and conclude that $\mu\le B_+-\lambda_+\theta$.  
Similarly a
lower bound can be found by choosing $\lambda_-$ such
that $\lambda_-h-f$ is a completely monotone function
of $\bsx$, using Theorem~\ref{thm:basic} to get
an upper bound $\lambda_-\theta-\mu\le B_-$
and then concluding that $\mu\ge\lambda_-\theta -B_-$.
Details on how to choose $h$ and find $\lambda_\pm$
are beyond the scope of this article.

The customary way to quantify uncertainty in QMC
is to use RQMC replicates with statistically derived
asymptotic confidence intervals.  For a recent
thorough empirical evaluation of RQMC, see
\cite{lecu:naka:owen:tuff:2023}, who found
the usual confidence intervals based on the central
limit theorem to be even more reliable than 
sophisticated bootstrap methods.
Here we have found
an alternative computable non-asymptotic approach with
100\% coverage, but so far
it does not give very good accuracy for high dimensions.

The error bounds that we produce do not
make any assumptions about regularity of $f$
beyond complete monotonicity. In particular,
smoothness of $f$ does not enter. For $d=1$ with
left and right endpoint rules to bracket $\mu$, 
the width of the uncertainty interval is
$(f(1)-f(0))/n$ regardless of any differentiability
properties of $f$. We leave it to future work to
determine what role if any smoothness plays for $d\ge2$.

\section*{Acknowledgments}
We thank Josef Dick, David Krieg, Frances Kuo, Dirk Nuyens
and Ian Sloan for discussions.
We also thank two anonymous reviewers for helpful
comments.
Much of this work took place at 
the MATRIX Institute's location in Creswick
Australia as part of their
research program on `Computational Mathematics for High-Dimensional Data in Statistical Learning', in February
2023, and the paper was finalized during the Dagstuhl Seminar 23351 `Algorithms and Complexity for Continuous Problems', in Schloss Dagstuhl, Wadern, Germany, in August 2023.
We are grateful to MATRIX and to the Leibniz Center Schloss Dagstuhl. 
The contributions of ABO and ZP were supported 
by the U.S.\ National
Science Foundation under grant DMS-2152780.
Peter Kritzer is supported by the Austrian Science Fund (FWF) Project P34808. For the purpose of open access, the authors have applied a CC BY public copyright licence to any author accepted manuscript version arising from this submission.
\bibliographystyle{plain}
\bibliography{qmc}

\end{document}